\documentclass[bj,preprint]{imsart}

\RequirePackage{amsthm,amsmath,amsfonts,amssymb}
\RequirePackage[numbers]{natbib}
\RequirePackage[colorlinks,citecolor=blue,urlcolor=blue]{hyperref}
\RequirePackage{graphicx}
\usepackage{enumerate} 
\usepackage{comment}
\usepackage{listings}
\usepackage{amsmath}
\usepackage{amsthm}
\usepackage{tikz}
\usepackage{caption}
\usepackage{array}
\usepackage{multirow}
\usepackage{mdwtab}
\usepackage{eqparbox}
\usepackage{multicol}
\usepackage{amsfonts}
\usepackage{multirow,bigstrut,threeparttable}
\usepackage{amsthm}
\usepackage{array}
\usepackage{bbm}
\usepackage{subfigure}
\usepackage{epstopdf}
\usepackage{mdwtab}
\usepackage{eqparbox}
\usepackage{latexsym}
\usepackage{bm}
\usepackage{amssymb}
\usepackage{graphicx}
\usepackage{mathrsfs}
\usepackage{epsfig}
\usepackage{psfrag}
\usepackage{algorithm}
\usepackage{algpseudocode}
\usepackage{stfloats}
\usepackage{algorithm}
\usepackage{mathrsfs}
\usepackage{graphicx}
\usepackage{tabularx,booktabs,siunitx,caption}
\newcolumntype{C}{>{\centering\arraybackslash}X}
\startlocaldefs
\numberwithin{equation}{section}
\theoremstyle{plain}
\newtheorem{theorem}{Theorem}
\newtheorem{example}{Example}
\newtheorem{lemma}{Lemma}
\newtheorem{remark}{Remark}
\newtheorem{corollary}{Corollary}

\newtheorem{definition}{Definition}

\def \bE {\mathbb{E}}
\def \bR {\mathbb{R}}

\def \var {\mathsf{Var}}

\newcolumntype{C}{>{\centering\arraybackslash}X}

\newcommand{\tabincell}[2]{\begin{tabular}{@{}#1@{}}#2\end{tabular}}

\usepackage{xspace}

\newcommand{\supp}{\mathrm{supp}}

\newcommand{\murej}{\hat{\mu}_{\mathrm{REJ}}}
\newcommand{\muimh}{\hat{\mu}_{\mathrm{IMH}}}

\newcommand{\Prob}{\mathbb{P}}

\newcommand{\TV}{{\sf TV}}

\newcommand{\Unif}{\mathsf{Unif}}

\definecolor{myblue}{rgb}{.8, .8, 1}
\definecolor{mathblue}{rgb}{0.2472, 0.24, 0.6} % mathematica's Color[1, 1--3]
\definecolor{mathred}{rgb}{0.6, 0.24, 0.442893}
\definecolor{mathyellow}{rgb}{0.6, 0.547014, 0.24}

\newcommand{\calX}{{\mathcal{X}}}

\newcommand{\dimh}{d_{\mathrm{IMH}}}

\usepackage{cleveref}
\crefname{lemma}{Lemma}{Lemmas}
\Crefname{lemma}{Lemma}{Lemmas}
\crefname{thm}{Theorem}{Theorems}
\Crefname{thm}{Theorem}{Theorems}
\crefformat{equation}{(#2#1#3)}

\DeclareMathOperator*{\arginf}{arg\,inf}

\startlocaldefs
\theoremstyle{remark}
\newtheorem*{fact}{Fact}
\endlocaldefs

\begin{document}

\begin{frontmatter}
\title{Exact Convergence Analysis of the Independent Metropolis-Hastings Algorithms}
%\title{A sample article title with some additional note\thanksref{t1}}
\runtitle{Exact Convergence Analysis of the IMH Algorithms}
%\thankstext{T1}{A sample additional note to the title.}

\begin{aug}
	%%%%%%%%%%%%%%%%%%%%%%%%%%%%%%%%%%%%%%%%%%%%%%
	%%Only one address is permitted per author. %%
	%%Only division, organization and e-mail is %%
	%%included in the address.                  %%
	%%Additional information can be included in %%
	%%the Acknowledgments section if necessary. %%
	%%%%%%%%%%%%%%%%%%%%%%%%%%%%%%%%%%%%%%%%%%%%%%
	\author{\fnms{Guanyang } \snm{Wang}\ead[label=e1]{guanyang.wang@rutgers.edu}}
	%%%%%%%%%%%%%%%%%%%%%%%%%%%%%%%%%%%%%%%%%%%%%%
	%% Addresses                                %%
	%%%%%%%%%%%%%%%%%%%%%%%%%%%%%%%%%%%%%%%%%%%%%%
	\address{Department of Statistics, Rutgers University,
		\printead{e1}}
\end{aug}

\begin{abstract}
A well-known difficult problem regarding Metropolis-Hastings algorithms is to get sharp bounds on their convergence rates. Moreover, a fundamental but often overlooked problem in Markov chain theory is to study the convergence rates for different initializations. In this paper, we study the two issues mentioned above of the Independent Metropolis-Hastings  (IMH)  algorithms on both general and discrete state spaces. We derive the exact convergence rate and prove that the IMH algorithm's different deterministic initializations have the same convergence rate. We get the exact convergence speed for IMH algorithms on general state spaces.
\end{abstract}

\begin{keyword}[class=MSC2010]
	\kwd[Primary ]{60J22}
	\kwd{62D05}
	\kwd[; secondary ]{62D05}
\end{keyword}

\begin{keyword}
	\kwd{Independent Metropolis-Hastings }
	\kwd{Markov chain Monte Carlo}
	\kwd{exact convergence rate}
\end{keyword}

\end{frontmatter}
%%%%%%%%%%%%%%%%%%%%%%%%%%%%%%%%%%%%%%%%%%%%%%
%% Please use \tableofcontents for articles %%
%% with 50 pages and more                   %%
%%%%%%%%%%%%%%%%%%%%%%%%%%%%%%%%%%%%%%%%%%%%%%
%\tableofcontents

\section{Introduction}\label{sec:introduction}

Markov chain Monte Carlo (MCMC) methods, such as the Metropolis-Hastings algorithms \cite{metropolis1953equation} \cite{hastings1970monte}, the Gibbs sampler \cite{geman1984stochastic} \cite{gelfand1990sampling}, have revolutionized  statistics, especially statistical computing and Bayesian methodology. When facing a complicated probability distribution, no matter discrete or continuous, low or high dimensional, MCMC methods provide  powerful simulation tools to draw samples from the target distribution. Mathematically speaking, given a target distribution $\pi(x)$, the MCMC algorithm generates a Markov chain $\Phi \doteq \{\Phi_1, \Phi_2, \cdots\}$ whose stationary distribution is designed to be $\pi$. Under mild conditions, classical Markov chain theory guarantees that the Markov chain $\Phi$ will converge to its stationary distribution when the chain runs long enough. 

However, the convergence theorem mentioned above  does not give any   quantitative convergence rate results. Therefore, a central part of the MCMC theory is to establish convergence rate analysis for MCMC algorithms. Unfortunately, getting sharp bounds for the convergence rates is known to be a formidable challenge  in Markov chain theory, especially for general state space Markov chains. Consider a Markov chain with stationary distribution $\pi(x)$ in $\bR^d$, most of the existing theoretical works are established under the notion of geometric ergodicity. Roughly speaking, a Markov chain with transition kernel $P$ is said to be geometrically ergodic if the following inequality holds for $\pi$-almost all $x$:
\begin{equation}\label{eqn:geometric ergodic}
\|P^n(x,\cdot) - \pi(\cdot)\| \leq C(x) r^n,
\end{equation}
where $r$ is a constant between $0$ and $1$, $\| \cdot \|$ is some distance metric between two probability measures, usually taken as the total variation (TV) distance.

Formula \eqref{eqn:geometric ergodic} guarantees that, for almost all initializations, the convergence rate  can be uniformly upper bounded by $r$. However, one can ask several natural questions regarding  Formula \eqref{eqn:geometric ergodic}:

\begin{itemize}
	\item  (Q1) \textit{How to get sharp convergence rate $r$  of inequality  \eqref{eqn:geometric ergodic}?} In practice, geometric ergodicity is often established under the `drift-and-minorization' framework. But the quantitative bound $r$ given by this framework is often very conservative and thus far from the exact convergence rate. For example, there exist a few cases where sharp rates of convergence can be derived, and the resulting bounds are usually much sharper than those based on `drift-and-minorization'.
	We refer the interested readers to the paper by Diaconis et al. \cite{diaconis2008gibbs}, the discussion on the previous paper by Jones and Johnson \cite{jones2008comment}, and the rejoinder \cite{diaconis2008}  for a concrete example. See also the recent work by Qin and Hobert \cite{qin2020limitations} for discussions on the limitations of the `drift-and-minorization' framework.
	
	\item  (Q2)  \textit{Does every point $x$ have the same convergence rate?} From the expression of inequality \eqref{eqn:geometric ergodic}, it is tempting to believe that all the points have the same convergence rate. However, as we will see shortly, this is not true in general.  Therefore, even if one could get a `sharp rate' for \eqref{eqn:geometric ergodic}, this still only tells the slowest convergence rate among all the initializations.  Suppose there exists a Markov chain, such that the convergence rate at one point (say $x_1$)  equals $0.001$, while  the convergence rate at another point (say $x_2$)  equals $0.999$. Then the bound $r$ given by  \eqref{eqn:geometric ergodic} would be practically useless when one starts the chain at $x_1$. Therefore, a more precise but  natural question would be inquiring about the convergence rates $r(x)$ for different initializations. 
\end{itemize}

We believe both Q1 and Q2 would be of both theorist and practitioner's interest. First of all, most of the present MCMC convergence theories focus on the upper bound of the convergence speed. Thus it is natural to ask for a sharp upper bound, or a useful lower bound. Secondly, an a priori bound for the convergence rate would be very helpful for users to evaluate the computation cost, or choose a fast MCMC algorithm among many candidates before implementing them all. Finally, but perhaps more importantly, the convergence speed analysis for different initializations is an interesting, important but overlooked question from both a mathematical and an algorithmic point of view. Mathematically, studying convergence rates for different initializations is quite natural after studying the worst-case convergence rates. Algorithmically, convergence analysis for different initializations   provides guidance for MCMC users directly, as a carefully chosen initialization can significantly speed up the algorithm. Unfortunately, it seems Q2 has not been seriously discussed and studied before (perhaps except the very recent work by  Lubetzky and Sly \cite{lubetzky2020fast} in the context of Ising Models).

However,  Q1 and Q2 are known to be technically  challenging. Sharp convergence rate analysis is  rare for MCMC algorithms on general state spaces, even for toy models. This paper provides precise answers to Q1 and Q2 in the context of the Independent Metropolis-Hastings (IMH) algorithm, a  widely-used class of  MCMC methods, on both discrete and general state spaces. We get the exact convergence rate of the IMH chain, and we show different initializations have the same convergence rate. Moreover, we derive the exact non-asymptotic convergence speed (in addition to the exact rate) for IMH algorithms on general state spaces. To the author's best knowledge, this is the first `exact convergence analysis' for certain MCMC algorithms on general state spaces, where the upper bound matches perfectly with the lower bound.

\subsection{Our contributions}\label{subsec:contribution}
In this paper, we consider  the IMH chain on general and discrete state space separately. Our target is to answer the questions Q1 and Q2 proposed in Section \ref{sec:introduction}. Our results are  summarized in
Table \ref{tab:summary_result}, and briefly explained below.

\begin{table}[htbp!]\caption{Summary of our convergence results for the IMH algorithms }\label{tab:summary_result}
	\begin{tabular}{c|c|c|c}
		\hline
		State Space $\mathcal X$ & \tabincell{c}{(Exact) Convergence \\Rate}  &\tabincell{c}{ (Exact) Convergence \\Speed} & \tabincell{c}{Convergence Rate for\\ Different Initializations} \\
		\hline
		Continuous               &  $1 - \frac 1 {w^\star}$  (Thm \ref{thm: exact convergence speed})& $(1 - \frac 1 {w^\star})^n$ (Thm \ref{thm: exact convergence speed})                    & All equals $1 - \frac 1 {w^\star}$ (Thm \ref{thm: every point convergence rate})            \\
		\hline
		Discrete                 &  $1 - \frac 1 {w^\star}$(Thm \ref{thm: exact convergence rate, discrete}) & \tabincell{c}{ $[c_1 (1- \frac 1{w^\star})^n ,  c_2(1- \frac 1{w^\star})^n]$\\ (Thm \ref{thm: exact convergence rate, discrete}) }                      & All equals $1 - \frac 1 {w^\star} $ (Thm \ref{thm: every point convergence, discrete})   \\       
		\hline  
	\end{tabular}
\end{table}

Let $\pi$ and $p$ be the probability density function (or the probability mass function is the state space is discrete) of the target distribution and the proposal distribution, respectively.  We assume the support of $\pi$ is contained in that of $p$, and the chain is always initialized from the support of $\pi$. We also assume the ratio between the target and proposal distribution is uniformly upper bounded, i.e., $w^\star \doteq \sup_x\frac{\pi(x)}{p(x)} < \infty$. Our results include: 
\begin{enumerate}
	\item  (Q1) Exact convergence rate analysis for the IMH algorithm:
	\begin{itemize}
		\item For the IMH algorithm on general state spaces, we prove the following `exact convergence' result:
		\begin{equation}\label{eqn:exact continuous state}
		\sup_{x\in\supp(\pi)} \|P^n(x,\cdot) - \pi\|_\TV   = (1 - \frac{1}{w^\star}) ^n.
		\end{equation}
	It is worth mentioning that formula \eqref{eqn:exact continuous state} completely characterizes the worst-case convergence speed for the IMH chain.
		\item For the IMH algorithm on discrete state spaces, we prove that:
		\begin{equation}\label{eqn:exact discrete state}
		c_1(1 - \frac{1}{w^\star}) ^n \leq  \sup_{x\in\supp(\pi)} \|P^n(x,\cdot) - \pi\|_\TV   \leq c_2(1 - \frac{1}{w^\star}) ^n,
		\end{equation}
		where $0 < c_1 \leq c_2 \leq1$ are two computable constants which will be defined in Section \ref{sec: IMH convergence, discrete}. 
	\end{itemize}
	\item(Q2) Convergence rate analysis with different initializations:
	\begin{itemize}
		\item For the IMH algorithm on both general and discrete state spaces, we prove that $P^n(x,\cdot)$ converges to $\pi$ at the same rate ($1 - \frac 1{w^\star}$) for all $x \in \mathcal X$ under certain conditions.
	\end{itemize}
\end{enumerate}

Theorem \ref{thm: exact convergence speed} gives a simple, clean formula for the exact speed of convergence, with the lower bound matches perfectly with the upper bound. Theorem \ref{thm: exact convergence rate, discrete} is the discrete variant of Theorem \ref{thm: exact convergence speed}. Interestingly, the exact result in Theorem \ref{thm: exact convergence speed} for general state space does not hold anymore when the state space is discrete, though the convergence rate is still the same. Theorem \ref{thm: every point convergence rate} and Theorem \ref{thm: every point convergence, discrete} are also somewhat surprising, as they show the IMH chain converges at the same rate for  arbitrary initializations, which gives  precise quantitative results on  convergence rates.

\subsection{Organization}\label{subsec: organization}
The rest of this paper is organized as follows. In Section \ref{sec:preliminaries}, we introduce our notations and review previous works. Section \ref{sec: IMH general} analyzes the IMH chain on general state spaces. Section \ref{sec: IMH convergence, discrete} analyzes the IMH chain on discrete state spaces. Some connections with the rejection sampler and the coupling from the past algorithms are discussed in Section \ref{sec: connections}.

\section{Preliminaries}\label{sec:preliminaries}

\subsection{The Metropolis-Hastings algorithm}
Let $(\mathcal X,\mathcal F, \pi) $ be a Polish space, equipped with $\sigma$-algebra $\mathcal F$ and a probability measure $\pi$. In practice, 
$\mathcal X$ usually stands for one of the following four possibilities:

\begin{enumerate}
	\item  \label{case: finite state} $\mathcal X$ is a  finite set. For example, $\mathcal X = \{x_1, x_2, \cdots, x_n\}$. 
	\item  \label{case: countable state} $\mathcal X$ is a countably infinite set. For example, $\mathcal X = \mathbb N$. 
	\item \label{case: Euclidean space} $\mathcal X$ is the Euclidean space, i.e., $\mathcal X = \mathbb \bR$ or $\bR^d$.
	\item \label{case: subset } $\mathcal X$ is a subset of the Euclidean space. For example, $\mathcal X = \bR^+$ or $\mathcal X = \{x\in \bR^d: \|x\|\leq c\}$ for some constant $c$.
\end{enumerate}
When $\mathcal X$ is a finite or countably infinite set, $\mathcal F$ is usually taken as $2^\mathcal{X}$ and thus all the subsets of $\mathcal X$ are measurable.  When $\mathcal X$ is the Euclidean space or a subset of the Euclidean space, $\mathcal F$ is usually taken as the Borel $\sigma$-algebra. The four possibilities of $\mathcal X$   cover almost all the practical applications of our interest. Therefore, without further specification, we shall assume the measurable space $\mathcal X$ falls into one of the four categories mentioned above. Throughout this paper,  $\mathcal X$ is called `state space'  as it  serves as the state space of the Markov chains. Naturally, we say `discrete state space' when $\mathcal X$ is finite or countably infinite, and `continuous state space' or `general state space' when $\mathcal X$ is the Euclidean space or a subset of the Euclidean space. In particular, when $\mathcal X$ is a general state space, we further assume the probability measure $\pi$ has a density on $\mathcal X$. We slightly abuse the notation and use $\pi$ to denote both the probability mass function (pmf) on discrete state spaces and the probability density function (pdf) on general state spaces. 

Given a target  distribution $\pi$, the Metropolis-Hastings (MH) algorithm is designed to draw samples from  $\pi$ approximately. The details of the Metropolis-Hastings algorithm are described in Algorithm \ref{alg:MHMC}.

\begin{algorithm}
	\caption{Metropolis-Hastings Algorithm (MH)}\label{alg:MHMC}
	\hspace*{\algorithmicindent} \textbf{Input:} initial setting $x$, number of iterations $T$, Markov transition kernel $q$ \\
	
	\begin{algorithmic}[1]
		
		\For{$t= 1,\cdots T$}
		\State Propose $x'\sim q(x,x')$
		\State Compute $a = \frac{q(x',x)\pi(x')}{q(x,x')\pi(x)} $
		\State Draw $r \sim \Unif[0,1]$
		\State \textbf{If} $(r< a)$ \textbf{then} set $x =x'$. \textbf{Otherwise}, leave $x$ unchanged.
		\EndFor
	\end{algorithmic}
\end{algorithm}

Algorithm \ref{alg:MHMC} is arguably the most popular class of Markov chain Monte Carlo (MCMC) methods. It is first proposed by   Metropolis et al.   \cite{metropolis1953equation} in 1953 to simulate a liquid in equilibrium with its gas phase. Hastings generalized the algorithm in 1970 \cite{hastings1970monte} and thus simulations following this scheme are called the Metropolis-Hastings algorithm. The MCMC algorithm can be implemented without the knowledge of the normalizing constant of the target distribution. This appealing feature makes MCMC methods very popular in many applications such as Bayesian inference.

To implement Algorithm \ref{alg:MHMC}, one needs to specify a  transition kernel $q$. Some common choices include the independent proposal,  the random-walk proposal, the Metropolis-adjusted Langevin algorithm (MALA). An MH algorithm using an independent proposal is called the Independent Metropolis-Hastings (IMH) algorithm, which means the proposal kernel  $q(x,y) = p(y)$ is independent of the current position.
The IMH algorithm is one of the most commonly used MCMC algorithms. Some modern variants and applications of IMH algorithm include the Adaptive IMH \cite{holden2009adaptive} and the Particle IMH \cite{andrieu2010particle}, \cite{middleton2019unbiased}. Moreover, IMH algorithms are  routinely used as a component of  auxiliary Monte Carlo methods, such as the Pseudo-marginal Monte Carlo sampler \cite{andrieu2009pseudo}, \cite{andrieu2015convergence}.

\subsection{Markov chain convergence}
Let $\Phi = \{\Phi_1, \Phi_2, \cdots \}$ be an aperiodic, irreducible Markov chain on state space $\mathcal X$ with transition kernel $P$. In the general state space setting, it is further assumed that $P(x,\cdot)$ is absolutely continuous with density $p(x,y)$ with respect to the Lebesgue measure, except perhaps for an atom $P(x,\{x\}) > 0$.

Standard Markov chain theory (see, for example, \cite{meyn2012markov} Chapter 13) guarantees that $\rVert P(x,\cdot) -\pi\lVert_\TV \rightarrow 0$ as $n\rightarrow \infty$ for $\pi$-almost everywhere $x\in \mathcal X$. Furthermore, a Markov chain $\Phi$  with transition kernel $P$ is said to be uniformly ergodic if 
\begin{definition}[Uniformly Ergodic]\label{def:uniform ergodic}
	\begin{equation}\label{eqn: uniformly ergodic}
	\sup_x \|P^n(x,\cdot) - \pi\|_\TV \leq C r^n
	\end{equation}
	for $C > 0$ and $0 < r < 1$.
\end{definition}

When $\mathcal X$ is a continuous, unbounded space, then a Markov chain  often fails to be uniformly ergodic. Therefore, the following concept `geometrically ergodic' is needed.  $\Phi$ is said to be geometrically ergodic if
\begin{definition}[Geometrically Ergodic]\label{def:geometrically ergodic}
	\begin{equation}\label{eqn:geometrically ergodic}
	\|P^n(x,\cdot) - \pi\|_\TV \leq C(x) r^n
	\end{equation}
	for some $0<r<1$, and $\pi$-almost all $x\in \mathcal X$.
\end{definition}
Both uniform and geometric ergodicity guarantee the Markov chain converges to its stationary distribution at a geometric rate for $\pi$-almost all initializations. However, it is a well-known difficult problem to give sharp bounds on the right-hand side of \eqref{eqn: uniformly ergodic} and \eqref{eqn:geometrically ergodic}.  In the next section, we will review some existing works on bounding the convergence rates of MH algorithms.
\subsection{Related works}
There are numerous studies concerning the convergence rates of Markov chains. The techniques for bounding convergence rates are significantly different between general and discrete state spaces.

When $\mathcal X$ is finite (but possibly very large), it is clear from Definition \ref{def:uniform ergodic} and \ref{def:geometrically ergodic} that every geometrically ergodic chain must be uniformly ergodic, as we could take the supremum of $C(x)$ in formula \eqref{eqn:geometrically ergodic} over the state space $\mathcal X$. There are many techniques,  including eigenanalysis, group theory, geometric inequalities, Fourier analysis, multicommodity flows, to get upper and lower bounds of the convergence rates.  These techniques can sometimes turn into very sharp bounds of \eqref{eqn: uniformly ergodic}, see Bayer and Diaconis \cite{bayer1992trailing}, Diaconis and Saloff-Coste \cite{diaconis1993comparison}\cite{diaconis1996logarithmic}, Diaconis and Shahshahani\cite{diaconis1981generating}, Liu \cite{liu1996metropolized} for inspiring examples. 
For general state space Markov chains,  uniform ergodicity is often `too good to be true' when the state space is  unbounded (see Theorem 3.1  of Mengersen and Tweedie \cite{mengersen1996rates}). Even geometrically ergodic is not at all trivially satisfied by usual Markov chains. General conditions for geometric ergodicity are established in the seminal works of Rosenthal \cite{rosenthal1995minorization}, Mengersen and Tweedie \cite{mengersen1996rates}, and Roberts and Tweedie \cite{roberts1996geometric}. More refined estimates for explicit $C(x)$ and $r$ are often referred to as `honest bounds' and are mostly developed under the `drift-and-minorization' framework  by Rosenthal \cite{rosenthal1995minorization}.  However, choosing a suitable `drift function' is often difficult in practical applications, and the estimates given by the `drift-and-minorization' framework are usually very conservative. Therefore, natural questions like Q1 and Q2 mentioned in the introduction are known to be challenging. For more about the convergence of  Markov chains on general state spaces, we refer the readers to an excellent survey by Roberts and Rosenthal \cite{roberts2004general}.

For IMH algorithms,  Liu \cite{liu1996metropolized} has computed  the eigenvalues and eigenvectors of the Markov transition matrix on a discrete state space. For  IMH algorithms on general state spaces, Mengersen and Tweedie have shown that uniform and geometric ergodicity are equivalent.  They have also provided the necessary and sufficient condition for the IMH algorithm being uniformly ergodic in  Theorem 2.1 of \cite{mengersen1996rates}. The exact transition probabilities for the IMH algorithm are derived by Smith and Tierney \cite{smith1996exact}. The whole spectrum of the IMH algorithm is derived by Atchad\'{e} and Perron \cite{atchade2007geometric}. Quantitative convergence bounds for non-geometrically ergodic IMH algorithms is studied by Roberts and Rosenthal  \cite{roberts2011quantitative}. The optimal scaling of the IMH sampler is studied by Lee and Neal \cite{lee2018optimal}.

\section{IMH on general state spaces}\label{sec: IMH general}
Now we are ready to state and prove our main results. Suppose one wishes to sample from a probability distribution with density $\pi$ on the general state space $\mathcal X$.  Let $p$ be a probability density function that we can sample from. Throughout this section, we assume both $p$ and $\pi$ are continuous. As mentioned in Section \ref{subsec:contribution}, we assume $\supp(\pi) \subset \supp(p)$ and the chain is initialized from $\supp(\pi)$.

The IMH algorithm is defined as follows: 

\begin{algorithm}
	\caption{Independent Metropolis-Hastings Algorithm (IMH)}\label{alg:IMH}
	\hspace*{\algorithmicindent} \textbf{Input:} initial setting $x$, number of iterations $T$, proposal distribution $p$ \\
	
	\begin{algorithmic}[1]
		
		\For{$t= 1,\cdots T$}
		\State Propose $x'\sim p(x')$
		\State Compute $a(x,x') = \frac{p(x)\pi(x')}{p(x')\pi(x)} $
		\State Draw $r \sim \Unif[0,1]$
		\State \textbf{If} $(r< a)$ \textbf{then} set $x =x'$. \textbf{Otherwise}, leave $x$ unchanged.
		\EndFor
	\end{algorithmic}
\end{algorithm}
This section is divided into two parts. In Section \ref{subsec: exact convergence speed}, we give the exact convergence speed of the IMH algorithm, which answers Q1 proposed in Section \ref{sec:introduction}. In Section \ref{subsec: converge every point}, we provide the answer of Q2 for the IMH algorithm.

\subsection{Exact convergence speed}\label{subsec: exact convergence speed}
Let $w$ be the ratio betwen densities $\pi$ and $p$, i.e.,
$w(x)  = \frac{\pi(x)}{p(x)}$. Then it is clear that the acceptance ratio $a$ can be written as 
$
a(x, x') = \frac{w(x')}{w(x)}.
$
The following theorem is proved by Mengersen and Tweedie, which gives the necessary and sufficient conditions for IMH being geometrically ergodic. 

\begin{theorem}[Mengersen and Tweedie \cite{mengersen1996rates}, Theorem 2.1]\label{thm: Mengersen Independence}
	With all the notations as above,  we have
	\begin{enumerate}
		\item Let $\Phi \doteq \{\Phi_1, \Phi_2, \cdots\}$ be the Markov chain associated with the IMH algorithms, with target density $\pi$ and proposal density $p$. Then the followings are equivalent:
		\begin{itemize}
			\item $\Phi$ is uniformly ergodic,
			\item $\Phi$ is geometrically ergodic,
			\item Function $w(x)$ is unformly bounded, i.e., there exists $M > 0$ such that $w(x) \leq M $ for all $x \in \mathcal X$.
		\end{itemize}
		\item Suppose $w(x) \leq M $ for all $x$ , then for any $n\geq 1$, we have
		\begin{equation}\label{eqn:imh_upperbound}
		\|P^n(x,\cdot) - \pi\|_\TV \leq (1- \frac 1M)^n
		\end{equation}
			for every $x$.
	\end{enumerate}
\end{theorem}

Theorem \ref{thm: Mengersen Independence} shows that $w$ being uniformly upper bounded is the necessary and sufficient condition for the IMH algorithm being geometrically ergodic. Therefore, we  assume $w^\star \doteq \sup_x w(x) < \infty$ henceforth. 
 In particular, uniformly ergodic and geometrically ergodic are equivalent for the IMH chain. It is therefore convenient to define:
\begin{definition}[Maximal Variation Distance]\label{def:total_variation_kernel}
	\[\dimh(t) \doteq  \sup_{x \in \supp(\pi)  } \|P^t(x,\cdot) - \pi\|_\TV.
	\]
\end{definition}
The maximal variation distance  $\dimh(t)$ describes the distance to the equilibrium distribution in the worst-case scenario. For the IMH chain, an immediate corollary for Theorem \ref{thm: Mengersen Independence} is:
\begin{corollary}\label{cor: IMH upper bound}
For any $n\geq 1$, we have
	\begin{equation}\label{eqn:imh_upperbound}
	\dimh(n) \leq (1- \frac 1{w^\star})^n.
	\end{equation} 
\end{corollary}

Corollary \ref{cor: IMH upper bound} gives the best possible upper bound of Theorem \ref{thm: Mengersen Independence}. Our next theorem shows, $(1- \frac 1{w^\star})^n$ is indeed a lower bound and is thus the exact convergence speed of the IMH chain. 

\begin{theorem}\label{thm: exact convergence speed} 
	With all the notations as above,   the exact convergence speed of the IMH chain is given by:
	\begin{equation}\label{eqn: exact convergence speed}
	\dimh(n) =   (1- \frac 1{w^\star})^n.
	\end{equation}
	In particular, assume that there exists a point which attains $w^\star$, i.e., there exists $x^\star\in \mathcal X$ such that $w(x^\star) = w^\star$, then we have:
	\begin{equation}\label{eqn: exact convergence speed, continued}
	\dimh(n) =   \|P^n(x^\star,\cdot) - \pi\|_\TV = (1- \frac 1{w^\star})^n.
	\end{equation}
\end{theorem}
It suffices to assume $w^\star > 1$ henceforth, as $w^\star = 1$ means  we are sampling from $\pi$ directly and $\dimh(n) = 0 = (1 - 1/w^\star)^n$ automatically holds. We will show both $\dimh(n) \leq   (1- \frac 1{w^\star})^n$ and $\dimh(n) \geq   (1- \frac 1{w^\star})^n$ holds. Though the first part is already shown in Mengersen and Tweedie \cite{mengersen1996rates}, we provide give a coupling proof here for  the sake of completeness.
\begin{proof}[Proof of the upper bound in Theorem \ref{thm: exact convergence speed}]
	
	We prove the upper bound by a coupling argument. We define a coupling of the IMH algorithm to be a pair of two stochastic processes $(\Phi, \tilde \Phi)\doteq (\Phi_i, \tilde \Phi_i)_{i= 0}^\infty$ such that both $\Phi$ and $\tilde \Phi$ are both IMH chains with proposal density $p$, though they may have different initial distributions. We further require the two chains stay together after they visit the same state, i.e.:
	\begin{equation}\label{eqn: coupling}
	\text{If } \Phi_i = \tilde\Phi_i  \text{ for some } i, \text{ then for any } j \geq i:
	\Phi_j = \tilde \Phi_j.
	\end{equation}
	Moreover, suppose $\tilde \Phi$ starts at stationary, in other words, 
	\begin{equation}\label{eqn: stationary initialization}
	\tilde \Phi_0 \sim \pi,
	\end{equation}
	then it is clear that  $\tilde\Phi_i $ is also distributed according to $\pi$. Then the following inequality is standard in Markov chain theory, see Levin et al. \cite{levin2017markov} Theorem 5.2, or Rosenthal \cite{rosenthal1995minorization} page 16 for a proof. 
	\begin{fact} Let $(\Phi, \tilde \Phi)$ be an arbitrary coupling satisfying \ref{eqn: coupling} and \ref{eqn: stationary initialization}. Let $T$ be the first time the two chains meet:
		\begin{equation}\label{eqn:couple time}
		T \doteq \min\{n, \Phi_n = \tilde \Phi_n\}.
		\end{equation}
		Then we have:
		\begin{equation}\label{eqn:coupling inequality}
		\|P^n(x,\cdot) - \pi\|_\TV \leq \Prob(T \geq n).
		\end{equation} 
	\end{fact}
	
	To use  inequality \ref{eqn:coupling inequality}, it suffices to design a coupling between $\Phi$ and $\tilde \Phi$ and bound the tail probability of meeting time  $T$.  
	
	In our IMH algorithm example, given any current state $x$, the actual transition density from $x$ to $y$ is given by:
	\[
	P(x,y) = p(y) \min\{\frac{w(y)}{w(x)}, 1\}\geq \frac{ p(y)w(y)}{w^\star} = \frac{\pi(y)}{w^\star}.
	\] 
	Therefore, the transition kernel $P$ satisfies the following minorization condition:
	\begin{equation}\label{eqn:minorization}
	P(x, \cdot) \geq \frac 1 {w^\star} \pi(\cdot)
	\end{equation}
	for any $x\in\mathcal X$. 
	
	Inequality \eqref{eqn:minorization} allows us to split $P(x,\cdot)$ into a mixture of two probability distributions:
	\begin{equation}\label{eqn:mixture density}
	P(x,\cdot) = \frac{1}{w^\star} \pi(\cdot) + (1 -  \frac{1}{w^\star}) q_{\text{res}}(x,\cdot),
	\end{equation}
	where $q_{\text{res}}(x,\cdot)$ is the residual measure defined by
	$
	q_{\text{res}}(x, A) = \frac{\pi(A) - \frac 1 {w^\star}P(x, A)}{ 1 -  \frac{1}{w^\star} } 
	$
	for any measurable set $A$.
	
	Now we are ready to define our coupling. At time $k$, let $x, x'$ be two arbitrary current states of $\Phi$ and $\tilde \Phi$, respectively. Assume $x' = x$, then sample $\Phi_{k+1} $ from $P$ and set $\Phi_{k+1} = \tilde \Phi_{k+1}$. Otherwise, flip a coin with head probability $\frac {1}{w^\star}$, if it comes up heads,  sample $\Phi_{k+1} $ according to $\pi$ and set $\Phi_{k+1} = \tilde \Phi_{k+1}$. Otherwise, sample $\Phi_{k+1}$ and $\tilde\Phi_{k+1}$ from $q_{\text{res}}(x, \cdot)$ and $q_{\text{res}}(x', \cdot)$ independently. 
	It is not hard to verify the pair $(\Phi, \tilde \Phi)$ defined above is a coupling satisfying \ref{eqn: coupling} and \ref{eqn: stationary initialization}. Meanwhile, at each step before $\Phi$ and $\tilde \Phi$ meets, the meeting probability equals $\frac {1}{w^\star}$. Therefore the following inequality holds:
	\begin{equation}\label{eqn: meeting time}
	\Prob(T \geq n) \leq (1- \frac 1{w^\star})^n.
	\end{equation} 
	Combining \ref{eqn:coupling inequality} with \ref{eqn: meeting time}, we have
	\begin{equation}\label{eqn: upper bound}
	\dimh(n) \leq   (1- \frac 1{w^\star})^n,
	\end{equation}
	as desired.
\end{proof}

The lower bound  requires the following two lemmas. The first lemma shows the total variation distance is lower bounded by a geometric series with  the `rejection probability' as the common ratio.

\begin{lemma}\label{lem: lower bound}
	Let $R(x)$ be the rejection probability of the IMH chain given current state $x$, more precisely, 
	\begin{equation}\label{eqn:rejection probability}
	R(x) = \int_\mathcal X (1 - \min\{\frac{w(y)}{w(x)}, 1\}) p(y) dy.
	\end{equation} 
	Then
	\begin{equation}\label{eqn:tv lower bound}
	\|P^n(x,\cdot) - \pi\|_\TV \geq (R(x))^n.
	\end{equation}
\end{lemma}
\begin{proof}[Proof of Lemma \ref{lem: lower bound}]
	Given a Markov chain starting at $x$, then it is clear that the probability of staying at $x$ after $n$ steps is no less than $(R(x))^n$.
	Therefore, we have
	\[P^n(x,\{x\}) \geq (R(x))^n.\]
	Meanwhile, it is clear that $\pi(\{x\}) = 0$ as $\pi$ is a continuous distribution. We have:
	\[
	\|P^n(x,\cdot) - \pi\|_\TV  =  \max_{A\subset \mathcal X} \|P^n(x,A) - \pi(A) \|\geq   (R(x))^n -  0 =  (R(x))^n,
	\] 
	as desired.
\end{proof}
The next lemma gives the formula for computing $R(x)$:

\begin{lemma}\label{lem: rejection probability}
	With all the notations as above, we have
	\begin{equation}\label{eqn: compute rejection probability}
	R(x) = 1 - \frac{\int_{y \in C(w(x))}\pi(y) dy } {w(x)} - \int_{y\in C^c(w(x))} p(y) dy ,
	\end{equation}
	where $C(s) \doteq \{x\in\calX : w(x) \leq s\} = w^{-1}((-\infty,s])$ is a subset of $\mathcal X$ which contains all the points with ratio no larger than $s$.
	
	In particular,  if there exists $x^\star\in \mathcal X$ such that $w(x^\star) = w^\star$, then
	\begin{equation}\label{eqn: largest rejection probability}
	R(x^\star) = 1 - \frac{1}{w^\star}.
	\end{equation}
\end{lemma}
\begin{proof}[Proof of Lemma \ref{lem: rejection probability}]
	For fixed $x$, we could split $\mathcal X$  into $C(w(x))$ and  $C^c(w(x))$, straightforward calculation gives:
	\begin{align*}
	R(x) &= \int_\mathcal X (1 - \min\{\frac{w(y)}{w(x)}, 1\}) p(y) dy \\
	& = 1 - \frac{\int_{y \in C(w(x))}p(y)  w(y)dy }{w(x)} - \int_{y\in C^c(w(x))} p(y) dy \\
	& = 1 - \frac{\int_{y \in C(w(x))}\pi(y) dy } {w(x)} - \int_{y\in C^c(w(x))} p(y) dy .
	\end{align*}
	In particular, when $x^\star$ exists, $C(w(x^\star)) = C(w^\star) = \mathcal X$ by definition, therefore 
	\[
	R(x^\star) = 1 - \frac{\int_{y\in \mathcal X} \pi(y) dy }{w^\star}  = 1 - \frac{1}{w^\star},
	\]
	which completes our proof.
\end{proof}
The lower bound of Theorem \ref{thm: exact convergence speed} is immediate after combining Lemma \ref{lem: lower bound} and Lemma \ref{lem: rejection probability}.
\begin{proof}[Proof of the lower bound in Theorem \ref{thm: exact convergence speed}]
	First we assume  there exists $x^\star\in \mathcal X$ such that $w(x^\star) = w^\star$. Combining Lemma \ref{lem: lower bound} and Lemma \ref{lem: rejection probability}, we have:
	\begin{align*}
	\dimh(n)  =   \sup_{x \in \supp(\pi)  } \|P^n(x,\cdot) - \pi\|_\TV \geq \|P^n(x^\star,\cdot) - \pi\|_\TV 
	\geq R^n(x^\star) = ( 1 -\frac 1 {w^\star})^n,
	\end{align*}
	which proves \ref{eqn: exact convergence speed, continued}. 
	
	When $x^\star$ is not guaranteed to exist, we use the standard technique to approximate $w^\star$ by a series of $w(x_n)$. More precisely,  For any $\epsilon > 0$, choose $\delta > 0$ such that $\delta < \frac{\epsilon (w^\star)^2}{1 + \epsilon w^\star}$. By definition, there exists
	$x_\delta$ such that $w(x_\delta) > w^\star - \delta$. Then $R(x_\delta)$  can be bounded by \eqref{eqn:rejection probability}
	\begin{equation}\label{eqn: rejection probability inequality}
	R(x_\delta) \geq \int_\mathcal X (1 - \frac{w(y)}{w(x_\delta)}) p(y) dy = 1 - \frac {1}{w(x_\delta)} > 1 - \frac{1}{w^\star - \delta} \geq (1 - \frac 1 {w^\star} - \epsilon).
	\end{equation}
	By Lemma \ref{lem: lower bound}, we have:
	\begin{align*}
	\dimh(n) & =   \sup_{x \in \supp(\pi) }  \|P^n(x,\cdot) - \pi\|_\TV  \geq\|P^n(x_\delta,\cdot) - \pi\|_\TV \geq  R^n(x_\delta)
	 = (1 - \frac 1 {w^\star} - \epsilon)^n.
	\end{align*}
	Since the above inequality holds uniformly for $\epsilon > 0$, we let $\epsilon \rightarrow 0$ and get 
	\[
	\dimh(n)  \geq \inf_{\epsilon > 0 } (1 - \frac 1 {w^\star} - \epsilon)^n = (1 - \frac 1 {w^\star})^n,
	\]
	which proves \eqref{eqn: exact convergence speed}, as desired.
\end{proof}
\begin{remark}\label{rem: general state convergence}
	When $x^\star$ exists, another way of showing the lower bound is as follows. We  observe that \begin{equation}\label{eqn: measure decomposition}
	P(x^\star, \cdot) = R(x^\star)\delta_{x^\star} + \frac {1}{w^\star} \pi(\cdot),\end{equation} 
	where $\delta_{x^\star} $ is the point mass at $x^\star$. As $\pi$ is invariant under the action of $P$, i.e., $\pi P = \pi$, applying the transition kernel $P$ on \eqref{eqn: measure decomposition} $n$ times yields:
	\[
	P^n(x^\star, \cdot) =  R^n(x^\star) \delta_{x^\star} + (1 -  R^n(x^\star))\pi(\cdot),
	\]
	and therefore $ \|P^n(x^\star,\cdot) - \pi\|_\TV = R^n(x^\star) \| \delta_{x^\star} - \pi\|_\TV = ( 1 -\frac 1 {w^\star})^n$.
\end{remark}

Theorem \ref{thm: exact convergence speed} gives the exact convergence speed as the upper bound matches with the lower bound exactly. This result closes the gap in the literature regarding the convergence speed of IMH algorithm.  For example,  Smith and Tierney \cite{smith1996exact} proved that $\dimh(t)$ is upper bounded by  a term of order at most $ ( 1 -\frac 1 {w^\star})^n$ (see page 5 of \cite{smith1996exact}). A similar result is also proved by Tierney (\cite{tierney1994markov}, Corollary 4). Mengersen and Tweedie (\cite{mengersen1996rates}, Theorem 2.1) proves $\dimh(t)\leq ( 1 -\frac 1 {w^\star})^n $. But none of the previous results contain the lower bound.

Now we provide some concrete examples to help illustrate the use of Theorem \ref{thm: exact convergence speed}. The next example is discussed in Smith and Tierney \cite{smith1996exact}, and Jones and Hobert \cite{jones2001honest}.

\begin{example}[Exponential target with exponential proposal]\label{eg: exponential}
	Suppose  the target distribution is $\text{Exp}(1)$, that is, $\pi(x) = e^{-x}$ and $\mathcal X  = [0,\infty)$. Consider the IMH chain $\Phi$ with an $\text{Exp}(\theta)$ proposal, that is, $p(x) = \theta e^{-\theta x}$. Therefore the weight function has the following form
	\[
	w(x) = \frac{\pi(x)}{p(x)} = \frac{1}{\theta} e^{-(1-\theta)x}.
	\]
	It is not hard to verify that $\sup w(x) <\infty$ is equivalent to $0 < \theta\leq 1$. Therefore, when $\theta \in (0,1]$, the algorithm is uniformly ergodic; otherwise, it is not geometrically ergodic. 
	
	For any fixed $\theta \in (0,1]$, we have $w^\star(\theta) = \frac{1}{\theta}$ which is attained at $x = 0$. Therefore we could apply Theorem \ref{thm: exact convergence speed} and conclude that  $
	\dimh(\theta, n) = (1 - \theta)^n.$
	This also shows the `optimal choice' of $\theta$ is $\theta = 1$, which is obvious  as $\theta = 1$ corresponds to sampling from the stationary distribution $\mathsf{Exp}(1)$ directly.
	
	For any arbitrary fixed $\theta$ and  accuracy $\epsilon$, we could also solve the equation $
	\epsilon =  (1 -\theta)^n$
	for $n$ to derive the number of steps for the IMH chain to converge  within $\epsilon$-accuracy. In this case we have 
	\begin{equation}\label{eqn:number of steps, exponential IMH}
	n = \frac{\log\epsilon}{\log(1- \theta)}
	\end{equation}
	steps are both necessary and sufficient. Let the accuracy $\epsilon$ be fixed at  $0.01$, then \eqref{eqn:number of steps, exponential IMH} gives us $n = 6.64$ when $\theta = 0.5$; $n = 43.71$ when $\theta = 0.1$; and $n = 458.21$ when $\theta = 0.01$. The relationship between $n$ and $\theta$ is also shown in Figure \ref{fig:exponential_exponential}. As expected, when $\theta$ is getting closer to $1$, the IMH algorithm is converging faster. 
	
	\begin{figure}[H]
		\includegraphics[width=0.5 \textwidth]{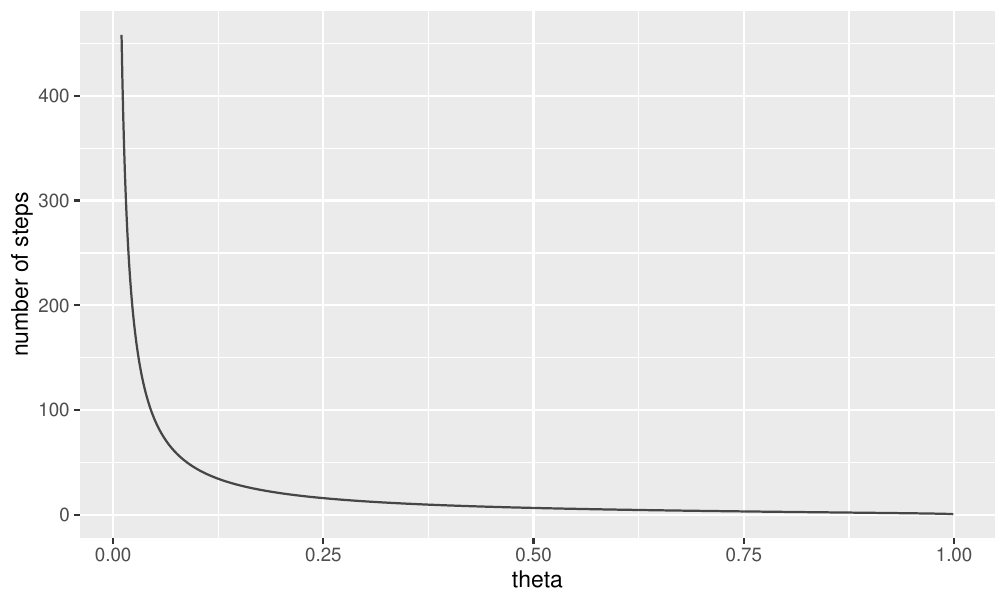}
		\caption{Number of steps  for the IMH algorithm to converge, $\epsilon = 0.01$}
		\label{fig:exponential_exponential}
	\end{figure}
\end{example}

Now we briefly conclude this part, we provide the answer to our first question (Q1) proposed in Section \ref{sec:introduction}   for the IMH algorithm on general state spaces. Our results can be briefly summarized below: 

\begin{enumerate}
	\item If $w^\star  = \infty$, then the IMH chain is not geometrically ergodic, thus the chain does not converge at a geometric rate.
	\item If $w^\star < \infty$  then we can derive the following exact converge speed:
	\[	\dimh(n) =   (1- \frac 1{w^\star})^n.\] In particular,
	if $w^\star$ can be attained at some $x^\star \in \mathcal X$, then the worst-case convergence speed happens when initializing the chain at $x^\star$.
\end{enumerate}

\subsection{Convergence rate at every point }\label{subsec: converge every point}

In  Section \ref{subsec: exact convergence speed}, we have shown the maximum distance between the IMH chain and its stationary distribution decays at the rate of $(1- \frac 1{w^\star})$. However, as we mentioned in Section \ref{sec:introduction}, this only gives an upper bound for the convergence rate of among all possible initializations.  To be more precise, we define the following  functions to describe the convergence rate at a single point $x$. 

\begin{definition}\label{def:convergence rate}
	Let \[r_{-}(x) \doteq \liminf_{n\rightarrow \infty} \|P^n(x,\cdot) - \pi\|_\TV^{1/n} \]
	and
	\[r^{+}(x) \doteq \limsup_{n\rightarrow \infty} \|P^n(x,\cdot) - \pi\|_\TV^{1/n}.\]
	
	In particular, if $r_{-}(x) = r^{+}(x)$ for some $x\in \mathcal X$, we define the convergence rate at $x$ to be $$r(x) \doteq \lim_{n\rightarrow \infty} \|P^n(x,\cdot) - \pi\|_\TV^{1/n}.$$ 
\end{definition}

The following corollary is immediate using  Lemma \ref{lem: lower bound} and Theorem \ref{thm: exact convergence speed}.

\begin{corollary}\label{cor: bound on convergence rate}
	With all the notations as above, we have
	\begin{itemize}
		\item 	For every $x\in  \supp(\pi)$, $
		R(x) \leq r_{-}(x) \leq r^{+}(x) \leq 1 - \frac 1{w^\star},$
		where $R(x)$ is the rejection probability at $x$ (defined in Lemma \ref{lem: lower bound}).
		
		\item $\sup_x r^{+}(x) = \sup_x r_{-}(x)  = 1 - \frac 1 {w^\star}$.
		\item If there is some $x^\star \in \calX$ such that $w(x^\star) = w^\star$, then
		$r(x^\star) = 1 - \frac 1 {w^\star}$.
	\end{itemize}
\end{corollary}

Corollary \ref{cor: bound on convergence rate} gives a lower bound for $r_{-}(x)$ and an upper bound for $r^{+}(x)$. However, the lower bound does not match the upper lower unless $x = x^\star$.  It is therefore natural to ask whether $r(x)$ always exists, if so, what is the value of $r(x)$? Our next result shows, for the IMH algorithm, every initialization converges at the same rate, which is  $1 - \frac {1}{w^\star}$.

\begin{theorem}\label{thm: every point convergence rate}
	Suppose $w^\star < \infty$. If the IMH chain satisfies one of the following two assumptions:
	\begin{itemize}
		\item $w^\star = w(x^\star)$ for some $x^\star \in \mathcal X$, and both $p$ and $\pi$ are locally Lipschitz. 
		\item Both $\pi$ and $w$ are  Lipschitz continuous. Meanwhile there exists a sequence $\{x_k\}_{k=1}^\infty$ and a constant $c_0 > 0$ such that $\lim_{k\rightarrow \infty}w(x_k) = w^\star$ and $\pi(x_k) \geq c_0$ for every $k$.  
	\end{itemize}
	
	Then for every $x\in \supp(\pi)$:
	\begin{equation}\label{eqn: every point convergence rate}
	r(x) = r_{-}(x) = r^{+}(x) = 1 - \frac {1}{w^\star}. 
	\end{equation}
	In other words, the convergence rate for every $x$ equals $ 1 - \frac {1}{w^\star}$. 
\end{theorem}

To prove Theorem \ref{thm: every point convergence rate}, we need to first define some measure tranformations from $\mathcal X$ to  $\bR$. Let $\tilde \Pi$ and $\tilde P$ be two cumulative distribution functions (CDFs) on $\bR$ defined by: 
\begin{align*}
& \tilde \Pi(s) \doteq \pi(C(s))  = \int_{y\in C(s)} \pi(y) dy\\
& \tilde P(s) \doteq p(C(s)) = \int_{y \in C(s)} p(y) dy, 
\end{align*}
where $C(s) \doteq \{x\in\calX : w(x) \leq s\}$  is defined in Lemma \ref{lem: rejection probability}. The function $\tilde{\Pi}: \bR\rightarrow [0,1]$ maps a real number $s$ to the $\pi$-probability of $C(s)$, the set of all the points with weight no larger than $S$. Therefore $\tilde\Phi$ is non-decreasing and goes to $1$ as $s\rightarrow\infty$ and therefore defines a CDF on $\bR$. The same reason holds for $\tilde P$. We denote the corresponding probability measures  by  $\tilde \Pi(ds)$ and $\tilde P(ds)$. Furthermore, define a function $\lambda: \bR^{+} \rightarrow \bR$ as
\begin{equation}\label{eqn: lambda function}
\lambda(s) \doteq \int_{v\leq s}(1 - \frac{v}{s}) \tilde P(dv) = \tilde P(s) - \frac{\tilde\Pi(s)}{s}.
\end{equation}
Comparing \ref{eqn: lambda function} with \ref{eqn:rejection probability}, it is not hard to verify that $\lambda(w) = 1 - \frac 1 w$ for $w\geq w^\star$, and $\lambda(w(x))$ equals the rejection probability $R(x)$ for any $x\in \mathcal X$. 
We need the following result on the $n$-step transition probability of the IMH algorithm, see Theorem 1 of Smith and Tierney \cite{smith1996exact} for proof.
\begin{theorem}\label{thm: Smith Tierney}
	The $n$-step transition kernel for the IMH chain is given by:
	\begin{equation}\label{eqn: n-step transition}
	P^n(x, dy) = T_n(\max\{w(x), w(y)\})\pi(y) dy + R^n(w(x)) \delta_x(dy),
	\end{equation}
	where $T_n: \bR^+ \rightarrow \bR$ is defined by:
	\begin{equation}\label{eqn: T_n}
	T_n(w) = \int_w^\infty \frac{n \lambda^{n-1}(v)}{v^2} dv.
	\end{equation}
\end{theorem}

Using Theorem \ref{thm: Smith Tierney}, the proof of Theorem \ref{thm: every point convergence rate} is given as follows:
\begin{proof}[Proof of Theorem \ref{thm: every point convergence rate}]
	We first prove Theorem \ref{thm: every point convergence rate} under the first assumption. That is, both $p$ and $\pi$ are locally Lipschitz continuous, and $w^\star(x) = w(x^\star)$ for some $x^\star \in\mathcal X$. By Corollary \ref{cor: bound on convergence rate}, it suffices to prove $r_{-}(x) \geq (1 - \frac 1 {w^\star})$. We prove this by contradiction. Suppose there exists an $\epsilon >  0$ and some $x_0\in \mathcal X$ such that 
	\[
	r_{-}(x_0) < 1 - \frac 1 {w^\star-\epsilon} ,
	\] 
	then we claim $w(x_0) \leq  w^\star - \epsilon$. Otherwise we have
	\[
	R(x_0) \geq \int_\mathcal X (1 - \frac{w(y)}{w(x_0)}) p(y) dy = 1 - \frac {1}{w(x_0)} > 1 - \frac 1 {w^\star - \epsilon},
	\]
	which contradicts with $R(x_0) \leq r_{-}(x_0)$.

	Given $w(x_0) \leq w^\star - \epsilon$, for any $y$ with $w(y) \geq w^\star - \epsilon$, \eqref{eqn: n-step transition} gives us:
	\begin{equation}\label{eqn: transition x0}
	P^n(x_0, dy) = T_n(w(y))\pi(y) dy .
	\end{equation}
	
	Notice that, we have the following estimate for $T_n(w)$:
	\begin{align*}
	T_n(w) &= \int_w^{w^\star} \frac{n \lambda^{n-1}(v)}{v^2} dv + \int_{w^\star}^\infty \frac{n (1 - \frac 1 v)^{n-1}}{v^2} dv \\
	& = \int_w^{w^\star} \frac{n \lambda^{n-1}(v)}{v^2} dv  + 1 - (1-\frac 1{w^\star})^n\\
	& \leq n\lambda^{n-1}(w^\star) \int_w^{w^\star}\frac1{v^2} dv  + 1 - (1-\frac 1{w^\star})^n\\
	& = n(1-\frac 1{w^\star})^{n-1} \frac{w^\star - w}{w w^\star } + 1 - (1-\frac 1{w^\star})^n.\\
	\end{align*}
	Let $D_n\doteq \{y\in\mathcal X: w(y) > \max\{1, w^\star - \frac{w^\star - 1}{2n}\}\}$ be a subset of $\mathcal X$. Then for any $y\in D_n$,
	\[
	T_n(w(y)) \leq 1 - \frac12(1- \frac 1{w^\star})^n 
	\]
	Therefore, when $n$ is large enough such that $w^\star - \frac{w^\star - 1}{2n} > w^\star -\epsilon$, by \eqref{eqn: transition x0} we have
	\[
	\lVert P^n(x_0, \cdot) -\pi \rVert_\TV\geq \lvert P^n(x_0, D_n) -\pi(D_n) \rvert\geq \frac12 \pi(D_n)(1- \frac 1{w^\star})^n 
	\]
	Now we claim there exists a universal constant $c > 0$ such that
	\[
	\pi(D_n) \geq \frac c{n^d}
	\]
	where $d$ is the dimensionality of $\cal X$. If the claim is true, then $r_{-}(x_0) \geq \lim_n (\frac{c}{2n^d})^{1/n}(1 - \frac 1{w^\star})  =  1 - \frac 1{w^\star} $, which contradicts with our assumption, as desired. 
	
	To prove the claim, notice that the function $w$ is also locally Lipschitz at $x^\star$ as both $\pi$ and $p$ are locally Lipschitz. Therefore, there exists some small $c' >0$ such that $w(y) \geq w^\star - \frac{w^\star-1}{2n}$ when $ \lVert y - x^\star \rVert \leq \frac{c'} n$. We can choose  $n$ large enough such that both $w^\star - \frac{w^\star - 1}{2n} > 1$ and $\pi(y) >\frac{\pi(x^\star)}{2}$ when $ \lVert y - x^\star \rVert \leq \frac{c'} n$, then we have
	$$\pi(D_n) \geq \int_{\lVert y - x^\star \rVert \leq \frac{c'} n} \pi(y) dy \geq \frac {(c')^d \pi(x^\star )V_d(1)}{2n^d},$$
	where $V_d(1) = \pi^{d/2}/\Gamma(\frac d2 + 1)$ is the volume of the unit $d$-sphere. This concludes the proof of Theorem \ref{thm: every point convergence rate} under the first assumption.
	
	Under the second assumption, we still have 
	\[
	\lVert P^n(x_0, \cdot) -\pi \rVert_\TV\geq \lvert P^n(x_0, D_n) -\pi(D_n) \rvert\geq \frac12 \pi(D_n)(1- \frac 1{w^\star})^n 
	\]
	with $D_n$ defined as above. To lower bound $\pi(D_n)$, we choose a subsequence $\{x_{k_n}\}\subset \{x_k\}$ such that $w(x_{k_n}) > w^\star - \frac{w^\star - 1}{4n}$. From the Lipschitz continuity of both $\pi$ and $w$, we can find $c' > 0$ small enough, such that for every $y$ satisfying $\lVert y - x_{k_n}\rVert\leq \frac{c'}{n}$, $y\in D_n$ and $\pi(y) > c_0/2$. Therefore we have:
	$$\pi(D_n) \geq \int_{\lVert y - x_{k_n} \rVert \leq \frac{c'} n} \pi(y) dy \geq \frac {(c')^d c_0V_d(1)}{2n^d},$$
	where $V_d(1) = \pi^{d/2}/\Gamma(\frac d2 + 1)$ is the volume of $S^{d-1}$, the unit sphere in $\bR^d$. Again, we have $r_{-}(x_0) \geq 1 - \frac 1{w^\star} $, which contracdicts with our assumption, as desired. 
\end{proof}

We believe the extra assumption requiring $\pi$ and $p$  locally Lipschitz continuous is satisfied in almost all the practical assumptions. On the one hand, in most  practical situations,  both $\pi$ and $p$ are continuously differentiable, which implies locally Lipschitz continuous. On the other hand, we do not know if Theorem \ref{thm: every point convergence rate} is true in full generality. It is an outstanding open problem to prove (or disprove) $r(x) = 1 - \frac 1{w^\star}$ for all $x$ without the `locally Lipschitz' assumption. 

We can apply Theorem \ref{thm: every point convergence rate} to Example \ref{eg: exponential} to obtain the convergence rate at every point.
\begin{example}[Exponential target with exponential proposal, continued]\label{exponential, continue}
	With all the settings same as Example \ref{eg: exponential}, since both $p$ and $\pi$ are continuously differentiable and are thus locally Lipschitz, we apply Theorem \ref{thm: every point convergence rate} to conclude the convergence rate for every $x\in [0,\infty)$ equals $1- \theta$. In Section 5 of Smith and Tierney \cite{smith1996exact}, they have calculated the transition probability via formula \ref{eqn: n-step transition} explicitly, and have:
	\[
	\Prob(\Phi_n > y| \Phi_0 = x) = \bigg(1 + \frac{(1-\theta)^n}{n\theta}\bigg) e^{-y} + o\bigg(\frac{(1-\theta)^n}{n}\bigg).
	\]
	The dominant term of total variation distance is also  $(1-\theta)^n$, which agrees with our result.
\end{example} 

Now we conclude this section. In Section \ref{subsec: converge every point} we answer Q2 for the IMH algorithms on general state spaces. Our results suggest, under mild conditions, the IMH algorithm converges at  the same rate (which is $1 - \frac 1 {w^\star}$ as shown in Section \ref{subsec: exact convergence speed})  for every $x\in \mathcal X$. Combining Section \ref{subsec: exact convergence speed} with Section \ref{subsec: converge every point}, we can answer Q1 and Q2 together in the following way: For an IMH chain on a general state space, the exact converge rate equals $1 - \frac 1 {w^\star}$, and every point has the same convergence rate. In particular, if $w^\star$ can be attained at some $x^\star$, then the exact convergence speed for $x^\star$ equals $(1 - \frac 1 {w^\star})^n$.
\section{IMH on discrete state spaces}\label{sec: IMH convergence, discrete}
Let $\mathcal X = \{x_1, x_2, \cdots, x_n\}$ be a  discrete state space. Let $\pi$ and $p$ be the target and proposal probability mass functions (PMFs), respectively.  Again, we define the ratio function $w \doteq \frac{\pi}{p}$ the same as in Section \ref{sec: IMH general}. Without loss of generality, we can relabel the elements in $\mathcal X$ such that $w$ is non-increasing, i.e.,
\[
w^\star = w(x_1) \geq w(x_2) \geq \cdots \geq w(x_n).
\]
In other words, the supremum of function $w$ is attained at the first state. In particular, if $w^\star$ is attained at multiple states, we label the first state to be an arbitrary one of these which has the smallest $\pi$-probability.

The IMH algorithm has the following transition probability kernel:
\begin{equation}\label{eqn: markov kernel}
P(x_i,x_j) = \begin{cases}
p(x_j)\min\{1, \frac{w(x_j)}{w(x_i)}\} \qquad \text{if } j\neq i\\
p(x_i) + \sum_k p(x_k) \max\{0, 1 - \frac{w(x_k)}{w(x_i)}\}\qquad \text{if } j = i.
\end{cases}
\end{equation}

It is tempting to believe that all the results in Section \ref{sec: IMH general} still hold for discrete state space. However, our next result shows, though the exact convergence rate is still $(1- \frac 1{w^\star})$, Theorem \ref{thm: exact convergence speed} does not hold.

\subsection{Exact convergence rate} \label{subsec: exact speed, discrete space}

\begin{theorem}\label{thm: exact convergence rate, discrete} 
	With all the notations as above, we have:
	\begin{equation}\label{eqn: exact convergence rate, discrete}
	(1 - \pi(x_1)) (1- \frac 1{w^\star})^t \leq \dimh(t) \leq (1 - \inf_{x_i\in \calX} \pi(x_i))   (1- \frac 1{w^\star})^t
	\end{equation}
\end{theorem}
\begin{remark}
	It is worth mentioning that the upper bound is slightly stronger than the upper bound shown by Liu \cite{liu1996metropolized}, page 117. The lower bound is not included in Liu \cite{liu1996metropolized}.
\end{remark}
\begin{proof}
	The right part of inequality \eqref{eqn: exact convergence rate, discrete} is  proved using  the same coupling strategy as in Theorem \ref{thm: exact convergence speed}, except that the coupling starts at time $0$ instead of time $1$. Let $\Phi$ be the IMH chain initialized at an arbitrary fixed state $x\in\calX$, and $\tilde \Phi$ be the IMH chain starting at stationary. At time $0$, the probability that $\Phi_0 \neq \tilde \Phi_0$ equals $1 - \pi(x)$. Given $\Phi_0 \neq \tilde \Phi_0$,  the coupling is designed in the same way as Theorem \ref{thm: exact convergence speed}, therefore at every step the two chain meets with probability no less than $1 - \frac{1}{w^\star}$. Therefore,  the coupling inequality \eqref{eqn:coupling inequality} shows
		\begin{equation}
	\|P^t(x,\cdot) - \pi\|_\TV \leq \Prob(T \geq t) \leq  (1 - \pi(x))(1 - \frac 1{w^\star})^t.
	\end{equation}
	Taking supremum over $x\in \calX$ on both sides yields
	\[
	\dimh(t) \leq (1 - \inf_{x_i\in \calX} \pi(x_i))   (1- \frac 1{w^\star})^t,
	\]
	which concludes the right part of \eqref{eqn: exact convergence rate, discrete}.

	For the left part, by \eqref{eqn: markov kernel}, we immediately have:
	\[
	P(x_1, \cdot) = (1 - \frac{1}{w(x_1)} ) \delta_{x_1} + \frac 1 {w(x_1)} \pi.
	\]
	
	Therefore, using the fact $\pi P = \pi$, we have:
	\[
	P^t(x_1, \cdot) = (1 - \frac{1}{w(x_1)} )^t \delta_1 + (1 -  (1 - \frac{1}{w_1} )^t)\pi.
	\]
	Then we have
	\begin{align*}
	\dimh(t) & =   \max_{x \in \mathcal X } \|P^t(x,\cdot) - \pi\|_\TV  \geq \|P^t(x_1,\cdot) - \pi\|_\TV \\
	&\geq (1 - \frac{1}{w(x_1)} )^t \|\delta_{x_1} - \pi\|_\TV = (1 - \pi(x_1))(1 - \frac{1}{w(x_1)} )^t,
	\end{align*}
	which concludes the proof.
\end{proof}
The following corollary is immediate:
\begin{corollary}
	The convergence rate $r$  (defined in Definition \ref{def:convergence rate}) at  $x = x_1$ equals $1 - \frac{1}{w^\star}$.
\end{corollary}

In contrast to Theorem \ref{thm: exact convergence speed}, Theorem \ref{thm: exact convergence rate, discrete} shows the `maximal total variation distance' $\dimh$ has different behaviors between general and discrete state spaces. For chains on general state spaces, $\dimh(t) = (1 - \frac{1}{w_1} )^t$ can be calculated explicitly. For chains on discrete state spaces, however, we can only bound  $\dimh$ between $(1 - \pi(x_1)) (1- \frac 1{w^\star})^t$ and  $(1 - \inf_{x_i\in \calX} \pi(x_i))   (1- \frac 1{w^\star})^t$. The lower and upper bounds coincide only when $x_1 = \arginf_{x_i\in\calX} \pi(x_i)$. An intuitive way of understanding this discrepancy is:  in fact, the proof of Theorem \ref{thm: exact convergence rate, discrete} works  for both discrete and continuous state space. Therefore a common lower bound for $\dimh(t)$ is $(1 - \pi(\{x^\star\})) (1 - \frac{1}{w_1} )^t$ and a common upper bound is 
$
(1 - \inf_{x\in \calX} \pi(\{x\})) (1 - \frac{1}{w_1} )^t.
$
If $\mathcal X$ is a general state space, then any probability density on $\mathcal X$ is atomless. Therefore the lower bound matches the upper bound. However, for discrete state $\mathcal X$, $\pi(\{x^\star\}) = \pi{(x_1)}$  not necessarily equals $\inf_{x\in \calX} \pi(\{x\})$, therefore the bounds may not coincide.

It is worth mentioning that all the results in this subsection still hold (after suitable notation changes) when the state space is countably infinite. All the proofs are essentially the same. 
\subsection{Convergence rate at every point }\label{subsec: converge every point, discrete}
In this section, we focus on proving the following theorem, which shows the convergence rate $r(x)$ equals $1 - \frac 1 {w^\star}$ for every $x\in \mathcal X$. 

\begin{theorem}\label{thm: every point convergence, discrete}
	Let $\Phi$ be an IMH chain on discrete finite state space $\mathcal X = \{x_1, \cdots, x_n \}$. Then for every $x\in \mathcal X$:
	\begin{equation}\label{eqn: every point convergence, discrete}
	r(x) = r_{-}(x) = r^{+}(x) = 1 - \frac {1}{w^\star}. 
	\end{equation}
	In other words, the convergence rate for every $x$ equals $ 1 - \frac {1}{w^\star}$. 
\end{theorem}

To prove Theorem \ref{thm: every point convergence, discrete}, we first review the following result from Liu \cite{liu1996metropolized}, which gives all the eigenvalues and the corresponding eigenvectors of the Markov transition kernel $P$ (defined in \eqref{eqn: markov kernel}).

\begin{theorem}[Theorem 2.1 of Liu \cite{liu1996metropolized}]\label{thm: eigen Liu1996}
	Let $\Phi$ be the IMH chain with transition kernel $P$ given by \eqref{eqn: markov kernel}. Then all the eigenvalues of the transion matrix are:
	\begin{equation}\label{eqn:IMH eigenvalues}
	\lambda_0 = 1 > \lambda_1 \geq \lambda_2 \geq \cdots \geq \lambda_{n-1}\geq 0,
	\end{equation}
	where 
	\[
	\lambda_k = \sum_{i=k}^n (p(x_i) - \frac{\pi(x_i)}{w(x_k)} ).
	\]
	The right eigenvector corresponding to $\lambda_k$ ($k > 0$) is 
	\begin{equation}\label{eqn:IMH eigenvectors}
	v_k = (0, 0, \cdots, 0, S_\pi(k+1), -\pi(x_k), \cdots, -\pi(x_k))^\intercal,
	\end{equation}
	where there are $k - 1$ zero entries and $S_{\pi}(j) = \pi(x_j) + \pi(x_{j+1}) + \cdots + \pi(x_n)$ for every $j$. 
	Moreover, $\langle v_i, v_j \rangle_\pi= 0$ for $i \neq j$.
\end{theorem}

Now we are ready to prove Theorem \ref{thm: every point convergence, discrete}.

\begin{proof}[Proof of Theorem \ref{thm: every point convergence, discrete}]
	Let $f_k$ be the normalized eigenfunction of $v_k$, that is:
	\[
	f_k = v_k/\langle v_k, v_k\rangle_\pi.
	\]
	Then the spectral theorem of the transition matrix gives us (see 
	Lemma 12.16 of Levin et al. \cite{levin2017markov} for a proof):
	\begin{equation}\label{eqn: l2 bound}
	\bigg\lVert \frac{P^t(x, \cdot)}{\pi(\cdot)} - 1\bigg\rVert_{2,\pi}^2 = \sum_{k = 1}^{m-1} f_k(x)^2 \lambda_k^{2t}\geq f_1(x)^2 \lambda_1^{2t}
	\end{equation}
Notice that $\lambda_1 = 1 - \frac 1 {w_(x_1)} = 1 - \frac 1{w^\star}$ and $v_1 = (1-\pi(x_1), -\pi(x_1), \cdots, -\pi(x_1))$ does not have any zero entry. We conclude
\begin{equation}\label{eqn:l2 convergence rate}
\bigg\lVert \frac{P^t(x, \cdot)}{\pi(\cdot)} - 1\bigg\rVert_{2,\pi}  \geq c(\pi) (1 - \frac 1{w^\star})^t
\end{equation}
for every $x$, where $c(\pi)$ is a positive constant depending only on $\pi$. 
Now it remains to prove the  $L_1$ convergence rate  equals $1 - \frac 1 {w^\star}$ for every $x$. Notice that:
\begin{equation}\label{eqn: inequality l1 vs l_infty}
\|P^t(x ,\cdot) - \pi\|_\TV = \frac 12 \lVert \frac{P^t(x, \cdot)}{\pi(\cdot)} - 1\rVert_{1,\pi}  =  \frac 12 \sum_{y\in \mathcal X} \bigg\lvert\frac{P^t(x, y)}{\pi(y)} - 1\bigg\rvert\pi(y) \geq \frac{\pi_\star}{2}  \max_y  \bigg\lvert \frac{P^t(x, y)}{\pi(y)} - 1\bigg\rvert,
\end{equation}
where $\pi_\star \doteq \min_y \pi(y)$, 
and 
\begin{equation}\label{eqn: inequality l2 vs l_infty}
\bigg\lVert \frac{P^t(x, \cdot)}{\pi(\cdot)} - 1\bigg\rVert_{2,\pi}  =\sqrt{ \sum_{y\in \mathcal X} \bigg(\frac{P^t(x, y)}{\pi(y)} - 1\bigg)^2\pi(y)} \leq  \max_y  \bigg\lvert \frac{P^t(x, y)}{\pi(y)} - 1\bigg\rvert
\end{equation}

Combining \eqref{eqn:l2 convergence rate}, \eqref{eqn: inequality l1 vs l_infty}, and \eqref{eqn: inequality l2 vs l_infty}, we have:
\begin{equation}\label{eqn: l1 convergence rate}
\|P^t(x ,\cdot) - \pi\|_\TV \geq  \frac{\pi_\star}{2}  	\bigg\lVert \frac{P^t(x, \cdot)}{\pi(\cdot)} - 1\bigg\rVert_{2,\pi} \geq   \frac{\pi_\star c(\pi)}{2}  (1 - \frac 1{w^\star})^t
\end{equation}
for any $x\in \mathcal X$. 

Formula \eqref{eqn: l1 convergence rate} implies
\[
r_{-}(x) \geq 1 - \frac 1{w^\star}.
\]
Since Theorem \ref{thm: exact convergence rate, discrete} gives us $r^{+}(x) \leq 1 - \frac 1{w^\star}$, we have $r(x) =  1 - \frac 1{w^\star}$, as desired.
\end{proof}

Theorem \ref{thm: every point convergence rate} and Theorem \ref{thm: exact convergence rate, discrete} shows that  the IMH chain has the same convergence rate for every initialization $x\in \mathcal X$, regardless of the sample space is continuous or discrete. This answers Q2 in Section \ref{sec:introduction}.

\section{Connections with rejection sampling and coupling from the past}\label{sec: connections}
\subsection{Comparison between the IMH and the rejection sampler}
Our assumption $w^\star = \sup_x \frac{\pi(x)}{p(x)} < \infty$
implies the target distribution $\pi$ is uniformly enveloped by the function $w^\star p(x)$. Therefore one can directly implement the rejection sampling algorithm to draw $i.i.d.$ samples from $\pi$ as follows. At each step, one samples $X\sim p$ and $U\sim \Unif[0,1]$. Then the new sample is accepted if $U \leq 
\frac{\pi(X)}{w^\star p(X)}$, and is discarded otherwise. Running the rejection sampling algorithm $T$ times will return $N_T$ accepted samples which are $i.i.d.$ distributed from $\pi$. It can be shown  that each time the acceptance probability is precisely $\frac{1}{w^\star}$ and therefore $N_T$ follows a Binomial distribution with $T$ trails  and success probability $\frac{1}{w^\star}$. Suppose we are interested in estimating $\mu = \bE_\pi(h(X))$. Let $X_1, \cdots, X_{N_T}$ be the number of accepted samples by running the rejection sampler $T$  times and $Y_1, \cdots, Y_T$ be the samples by running the IMH algorithm for $T$ steps. Then we have the following two natural estimators, denoted by $\murej$ and $\muimh$ respectively:
\begin{align}\label{eqn:estimator}
\murej(T)\doteq \frac{\sum_{i=1}^{N_T}h(X_i)}{N_T} \qquad \muimh(T)\doteq \frac{\sum_{i=1}^{T}h(Y_i)}{T}.
\end{align}

It is clear that the rejection sampler gives less but $i.i.d.$ samples, while the IMH algorithm gives more correlated samples.  Liu \cite{liu1996metropolized} compared the efficiency of both estimators in \eqref{eqn:estimator} when the state space is finite.  Liu proved the asymptotic efficiency of $\muimh$ is comparable with (at most half of)  that of $\murej$ (see page 117--118 of \cite{liu1996metropolized}) by exact eigenanalysis. We extend the previous study to general state space IMH algorithms. 

Following Liu \cite{liu1996metropolized}, we also compare the two estimators in terms of their asymptotic variances, which is defined as:
\[
\sigma^2(\murej) \doteq \lim_{T\rightarrow\infty}  \var(\sqrt{T}\murej(T))
\]
and 
\[
\sigma^2(\muimh) \doteq \lim_{T\rightarrow\infty}  \var(\sqrt{T}\muimh(T))\footnote{It is assumed that the chain starts at the stationary distribution, i.e., $Y_1\sim \pi$.}.
\]

The asymptotic variance of $\murej$ is:
\[
\sigma^2(\murej) = \lim_{T\rightarrow \infty} T \cdot \var(\murej(T)) = \lim_{T\rightarrow \infty} \var_\pi(h)\cdot T\cdot  \bE(\frac{1}{N_T}) = w^\star \var_\pi(h),
\]
where the last equality follows from the law of large numbers.

The asymptotic variance of $\muimh$ can be calculated using the following formula on page 751 of Tan \cite{tan2006monte}:
\[
\sigma^2(\muimh)  = \int_{x\in\calX} \tilde h^2(x) \frac{1 + R(x)}{1 - R(x)} \pi(x) dx,
\]
where $\tilde h = h - \bE_\pi(h)$, $R(x)$ is the rejection probability at state $x$.  Since we know $R(x) \leq  1- \frac{1}{w^\star}$  and $\frac{1 + R(x)}{1 - R(x)}\geq 1$ for every $x\in \calX$, we conclude
\begin{align}\label{eqn: asymptotic variance imh}
\var_\pi(h) \leq \sigma^2(\muimh) \leq (2w^\star - 1) \var_\pi(h).
\end{align}
Therefore,  the asymptotic variance for the IMH algorithm is always no less than $i.i.d.$ samples from $\pi$ and is comparable to the asymptotic variance of the rejection sampler. An alternative way of proving \eqref{eqn: asymptotic variance imh} is to use the spectral measure representation of the asymptotic variance. We skip the details here and refer the readers to Geyer \cite{geyer1992practical} for further discussions.

In practice, however, the asymptotic variance of the rejection sampler can be much higher than $w^\star$, as $w^\star$ is the best possible enveloping constant and is usually intractable. Choosing a large enveloping constant will result in a smaller acceptance probability and thus a higher asymptotic variance. In contrast, implementing the IMH algorithm does not require the knowledge of $w^\star$.

\subsection{Connections with coupling from the past}
The condition $w^\star < \infty$  also allows one to apply the coupling from the past (CFTP) algorithm to draw samples from $\pi$ directly. The CFTP algorithm is first proposed by Propp and Wilson \cite{propp1996exact} and has then been a very active area for more than twenty years. The idea of CFTP relies on the stochastic recursion representation of Markov chains. That is, given current state $X_t$, the next state $X_{t+1}$ can be represented by $X_{t+1} = \phi(X_t, U_t)$ where $\phi$ is a deterministic function and $U_t$ are random numbers independent of all the other $X_s$ and $U_s$. In the context of IMH algorithms, $U_t$ can be chosen as a tuple $(U_{t,1}, U_{t,2})$ where $U_{t,1}\sim p$ is the proposed new state and $U_{t,2}\sim \Unif[0,1]$, $\phi(X_t, U_t)$ equals $U_{t,1}$ if $U_{t,2} \leq \frac{\pi(U_{t,1})p(X_t)}{\pi(X_{t})p(U_{t,1})}$ and equals $X_t$ otherwise. Moreover, drawing a sample from $\pi$ is conceptually equivalent to drawing a sample from the IMH chain after running it for infinitely many steps. Suppose we start  many Markov chains  which are initialized at every $x\in \calX$ at time $-\infty$, and  each chain uses the same bit of randomness (the same $U_t$) to update the move for each $t$. We can show that all the paths from different initializations will eventually coalesce into one almost surely, and the sample at time $0$ is distributed according to $\pi$. CFTP can be used for IMH algorithms since $x^\star$ (assume it exists) plays a special role in the updates. For any state $x\in \cal X$ and any proposal $y\in \calX$, the acceptance ratio always satisfies $\frac{w(y)}{w(x)}\geq \frac{w(y)}{w(x^\star)}$. In other words, given any proposed move for all the chains, if the chain at $x^\star$ agrees to move, all the paths merge into one simultaneously. The algorithm is described in page 493 of Murdoch and Green \cite{murdoch1998exact} and in page 303 of Corcoran and Tweedie \cite{corcoran2002perfect}, using slightly different languages. We also refer the readers to  \cite{wang2021regeneration}  for recent developments. As noticed by Murdoch and Green \cite{murdoch1998exact},  the first time that an IMH chain starting at $x^\star$ makes a move is the same as the first time one gets a sample from rejection sampling. Thus the CFTP algorithm for the IMH is simply the rejection sampler.

\section*{Acknowledgement}
The author would like to thank Richard Smith and Persi Diaconis for helpful discussions, and Jeffrey Rosenthal for pointing out some useful references, and David Sichen Wu and Zhengqing Zhou for helpful suggestions in improving this paper.  The author would like to thank the Editor, the Associate Editor, and two referees for many very helpful suggestions.
%% or include bibliography directly:
\newpage
\bibliographystyle{alpha}
\bibliography{mybib}

\end{document}